\documentclass[12pt,a4paper]{amsart}
\usepackage{graphicx,amssymb}
\input xy
\xyoption{all}
\usepackage[all]{xy}
\usepackage{hyperref}
\usepackage{mathtools}

\newtheorem{thm}{Theorem}[section]
\newtheorem{cor}[thm]{Corollary}
\newtheorem{lem}[thm]{Lemma}
\newtheorem{prop}[thm]{Proposition}
\theoremstyle{definition}
\newtheorem{defn}[thm]{Definition}
\newtheorem{rem}[thm]{Remark}

\newtheorem{hyp}[thm]{Hypothesis}
\numberwithin{equation}{section}

\newcommand{\bQ}{\mathbb Q}

\newcommand{\bZ}{\mathbb Z}
\newcommand{\ZZ}{\mathbb Z}

\newcommand{\lra}{\longrightarrow}

\newcommand{\ra}{\rightarrow}

\newcommand{\cO}{\mathcal{O}}

\newcommand{\inj}{\hookrightarrow}
\newcommand{\surj}{\mathrel{\mathrlap{\rightarrow}\mkern1mu\rightarrow}}

\DeclareMathOperator{\Pic}{Pic}

 \DeclareMathOperator{\Div}{{Div}}
\DeclareMathOperator{\End}{{End}}

\begin{document}

\title{Symmetric correspondences with decomposable minimal equation}
\author{Elham Izadi and  Herbert Lange}

\address{E. Izadi\\ Department of Mathematics\\UC San Diego\\USA}
\email{eizadi@math.ucsd.edu}

\address{H. Lange \\ Department Mathematik der Universit\"at Erlangen \\ Germany}
\email{lange@mi.uni-erlangen.de}

\thanks{Part of this work was done while the second author was visiting the University of California San Diego. He wishes to thank UCSD for its support and hospitality.}

\date{\today }
\begin{abstract} 
We study symmetric correspondences with completely decomposable minimal equation on smooth projective 
curves $C$. The Jacobian of $C$ then decomposes correspondingly.
For all positive integers $g$ and $\ell$, we give series of examples of smooth curves $C$ of genus $n^\ell (g-1) +1$ with correspondences satisfying minimal equations of degree $\ell+1$ such that the Jacobian of $C$ has at least $2^\ell$ isogeny components.
\end{abstract}

\maketitle

\section{Introduction}

A correspondence of a smooth projective curve $C$ into itself is a divisor on the product $C \times C$. Correspondences were
extensively studied by the Italian geometers of the 19th century. Their importance lies in the fact that they determine endomorphisms of the Jacobian $J(C)$ of $C$. In fact,
the ring of equivalence classes of these correspondences is isomorphic to the ring of endomorphisms of $J(C)$.  

One way to use  this is given by the well-known bigonal, trigonal and tetragonal (or, more generally, $n$-gonal) constructions,
where one uses correspondences to detect isomorphisms between certain Prym varieties and Jacobians.
Another way is to construct Prym-Tyurin varieties, that is, roughly speaking, abelian subvarieties of Jacobians 
to which the restriction of the canonical polarization is a multiple of a principal polarization.\\

Here we study symmetric completely decomposable correspondences, that is correspondences which satisfy a minimal equation 
over the rationals all of whose zeros are rational numbers. It is easy to see  (Proposition \ref{prop2.1}) that such a correspondence induces a decomposition
up to isogeny of the corresponding Jacobian into a product of abelian subvarieties which are given by the zeros of the minimal polynomial.

We give a series of examples for this: Consider a smooth projective curve $X$ of genus $g \geq 1$ with an \'etale cover $f: Y \ra X$ of degree $n\geq 2$. Given a set $\{\sigma_1, \dots, \sigma_\ell \}$ of automorphisms for any positive integer $\ell$, we let $C$ be the $\ell$-fold fibre product of $Y$ with itself over $X$ for the compositions $\sigma_if: Y \ra X$ for $i = 1, \dots, \ell$:
$$
C := Y_{\sigma_1} \times_X Y_{\sigma_2} \times_X \cdots \times_X Y_{\sigma_\ell},
$$
where $Y_{\sigma_i} := Y$ but with the \'etale cover $\sigma_i \circ f : Y_{\sigma_i} \ra X$.
The curve $C$ is always smooth and we show in Section \ref{secirred} that it is irreducible if and only if the monodromy group of the cover $f : Y \ra X$ is $\ell$-transitive.

The curve $C$ admits an effective correspondence $D$ of degree $\ell(n-1)$ given by
\[
\begin{split}
D:= \left\{ (((\sigma_1 x)_{j_1}, \ldots, (\sigma_\ell x)_{j_\ell}), ((\sigma_1 x)_{k_1}, \ldots, (\sigma_\ell x)_{k_\ell})) \in C \times C \mid \right.\\
\left. k_i = j_i \mbox{ except for exactly one index } i \right\}.
\end{split}
\]
Note that for $\ell = 1$ and $\sigma_1 = id$ we have $C =Y$ and $D$  is the well known correspondence
$$
D = \{(y_1,y_2) \;|\; y_1 \neq y_2, f(y_1) = f(y_2) \} \subset Y^2.
$$ 
Our main result is:

\begin{thm}
The correspondence $D$ is completely decomposable and satisfies the equation
$$
\prod_{r-0}^\ell (D -nr + \ell) = 0.
$$
\end{thm}
We also identify the eigen-abelian subvarieties of the Jacobian $JC$ for the eigenvalues $n(\ell -r) - \ell$ for all $r$ (Corollary \ref{corPol}). We compute their dimensions (Lemma \ref{lemdim}) and show that $C$ has principally polarized quotients $(B, \Xi)$ such that the image of $C$ in $B$ has cohomology class equal to $\ell ! n^\ell \frac{[\Xi]^{d-1}}{(d -1)!}$ (Lemma \ref{lemclass}).

It follows from our results that the Jacobian of $C$ is isogenous to the product of at least $2^\ell$ abelian varieties. To our knowledge, up to now, examples of such curves were only obtained from curves with sufficiently large automorphism groups (see, e.g., \cite{CLRR2009}). Starting with an arbitrary curve $X$ of genus $\geq 2$ and any \'etale cover $Y\ra X$ of degree $n\geq 2$, we can construct the curve $C$ with $\sigma_1 = \ldots = \sigma_\ell = Id$. In this way we obtain a family of dimension $3g-3$ of curves $C$ of genus $n^\ell (g-1)+1$ with at least $2^\ell$ isogeny factors.

In Section \ref{secprelim} we define completely decomposable correspondences and derive the corresponding decomposition of the Jacobian.
Section \ref{secmain} contains the proof of the main theorem and its consequences. In Section \ref{secspec} we mention some special cases. Finally, Section \ref{secirred} contains the statement and proof of the irreducibility criterion for $C$.
 
\section{Correspondences on a curve $C$}\label{secprelim}

Let $C$ be a smooth projective curve over an algebraically closed field of characteristic 0. A  
{\it correspondence of $C$ of bidegree} $(d_1,d_2) \in \ZZ^2$ is given by a divisor 
$D$ of $ C \times C$ 
such that the projections $p_i: D \ra C$ are of degree $d_i$ for $i = 1$ and 2.  
Here the degree $d_i$ is defined as follows: If $D =  \sum_k a_k D_k$ with reduced and irreducible curves 
$D_k \subset C \times C$, then $d_i = \sum_k a_k \deg(p_i: D_k \ra C)$.

The transposed correspondence of $D$ is by definition the correspondence
$$
D^t := \tau^*D
$$
where $\tau$ denotes the natural involution of $C \times C$ which swaps the factors.
A correspondence is called {\it symmetric} if $D^t = D$. 
For a symmetric correspondece we have $d_1 = d_2$ and $d := d_1$ is called its {\it degree}.

Any effective correspondence of bidegree $(d_1,d_2)$ gives an
algebraic map
\begin{equation}\label{eqDmap}
D: C \ra C^{(d_1)},    \quad   x \mapsto D(x) =p_{2*} p_1^*(x).
\end{equation}
Here, as usual, $C^{(d_1)}$ denotes the $d_1$-th symmetric power of $C$.
Since any divisor on $C\times C$ can be written in a unique way as the difference of 2 effective divisors, it is 
clear how to extend this definition to any correspondence.
 Often 
correspondences are defined via such maps. If we define
$$
D\left(\sum_k a_kx_k\right) := \sum_k a_k D(x_k),
$$
for a correspondence $D$ and any divisor 
$\sum_k a_k x_k$ of $C$, it is clear how to define the $n$-fold power $D^n$ of $D$ as a group homomorphism from the group of divisors $\Div C$ to itself.  Hence $\ZZ[D]\subset \End (\Div C)$ is a $\ZZ$-module.

\begin{defn}
Assume that $\bZ [D]\subset \End (\Div C)$ is finitely generated. Let $\varphi$ be the generator of the ideal of polynomials of $\bQ [X]$ vanishing on $D$ such that
\begin{enumerate}
\item $\varphi$ has integer coefficients and 
\item the leading coefficient of $\varphi$ is positive and minimal among the positive leading coefficients of the generators with integer coefficients.
\end{enumerate}
We call $\varphi$ the minimal polynomial of $D$.
\end{defn}

We say $\varphi$ is {\it completely decomposable} if all its zeros are integers, i.e., if
$$
\varphi = a \prod_{i=1}^k (X - a_i) \quad \mbox{with} \quad a, a_i \in \ZZ \quad \mbox{for all} \; i.
$$

Two correspondences $D_1$ and $D_2$ are equivalent if there exist line bundles $L$ and $M$ on $C$ such that (see \cite[Section 11.5]{cav})
\[
\cO_{C^2} (D_1) \cong \cO_{C^2} (D_2)\otimes p_1^* L \otimes p_2^* M.
\]
The set of equivalence classes of correspondences is a ring (see \cite[Exercise 11.12.14]{cav}) which, 
according 
to \cite[Theorem 11.5.1]{cav}, is isomorphic to the endomorphism ring $\End(JC)$ of the Jacobian of 
$C$. Hence, associating to $D$ its equivalence class defines a natural map
$$
\gamma: \ZZ[D] \lra \End(JC).
$$
We denote the image of $D$ by $\gamma_D$.

\begin{defn}
The minimal polynomial $\psi$ of $\gamma_D$ is the generator of the ideal of polynomials of $\bQ [X]$ vanishing on $\gamma_D$ such that
\begin{enumerate}
\item $\psi$ has integer coefficients and 
\item the leading coefficient of $\psi$ is positive and minimal among the positive leading coefficients of the generators with integer coefficients.
\end{enumerate}
\end{defn}

Clearly $\psi$ is a divisor of $\varphi$, hence also completely decomposable 
with pairwise different zeros. We may choose the indices in such a way that
$$
\psi(D) = b \prod_{i=1}^\ell (D - a_i )
$$
for $\ell \leq k$ and $b \mid a$.
For all $i = 1, \dots, \ell$ we define $A_i$ to be the component of the identity of the kernel of the endomorphism $\gamma_D - a_i$:
$$
A_i :=  \ker(\gamma_D - a_i)^0 \subset JC.
$$

\begin{prop} \label{prop2.1}
Suppose the effective correspondence $D$ of $C$ admits a completely decomposable minimal
polynomial. Then, 
with the notation above, the addition map
$$ 
\alpha: A_1 \times \cdots \times A_\ell \ra JC
$$
is an isogeny.
\end{prop}

\begin{proof}
The map $\alpha$ is an isogeny if and only if its differential is an isomorphism of tangent spaces at 0.
Now, the tangent space of $A_i$, as a subspace of the tangent space of $JC$ at 0, is the 
corresponding eigenspace 
of the differential of $\gamma_D$ at $0$. Since the $a_i$ are pairwise distinct, 
the tangent spaces of the $A_i$ give a decomposition of the tangent space of $JC$. So the differential of $\alpha$ is given by the eigenspace decomposition of the tangent space and is thus an isomorphism.
\end{proof}

\begin{rem}
Whereas a correspondence $D$ of $C$ does not necessarily admit a minimal polynomial equation, its image 
$\gamma_D \in \End(JC)$ does. Hence an analogous result to 
Proposition \ref{prop2.1} is valid in greater generality, expressing it  for endomorphisms instead of correspondences. We chose the above form, since this is exactly what we need for our applications. 
\end{rem}

\section{\'Etale covers of curves with automorphisms}\label{secmain}

\subsection{} Let $X$ be a smooth curve of genus $g$ with an \'etale cover
$$
f : Y \lra X,
$$
of degree $n \geq 2$. For an automorphism $\sigma$ of $X$, denote
\[
Y_\sigma := Y, \:\: \hbox{ but with the covering } \quad f_\sigma := \sigma \circ f : Y \lra X.
\]
Given a set $\{ \sigma_1, \ldots, \sigma_\ell \}$ of automorphisms of $X$, let $C_{\sigma_1, \ldots , \sigma_\ell}$ denote the fiber product:
\begin{equation}  \label{diag3.1}
\xymatrix{
C_{\sigma_1, \ldots , \sigma_\ell} := Y_{\sigma_1} \times_X \ldots  \times_X Y_{\sigma_\ell}.
}
\end{equation}
Note that the automorphisms are allowed to be the identity and need not be distinct. If we denote the fibre of $f$ over $x\in X$ by
$$
f^{-1}(x) = \{ x_1, \dots , x_n \},
$$
then, as a set,
$$
C_{\sigma_1, \ldots , \sigma_\ell} = \{ ((\sigma_1 x)_{j_1}, \ldots, (\sigma_\ell x)_{j_\ell}) \mid 1 \leq j_i \leq n \} \subset Y^{\ell}.
$$

Since $f$ is \'etale, so are $\sigma_i f$ and the projections $q_i : C_{\sigma_1, \ldots , \sigma_\ell} \ra Y_{\sigma_i}$ for $i = 1, \ldots, \ell$. Hence $C_{\sigma_1, \ldots , \sigma_\ell}$ is smooth. In general it is not irreducible.  In Section \ref{secirred} we show that $C$ is irreducible if and only if the monodromy group of the cover $f : Y \ra X$ is $\ell$-transitive.

For $((\sigma_1 x)_{j_1}, \ldots, (\sigma_\ell x)_{j_\ell})\in C_{\sigma_1, \ldots , \sigma_\ell}$, define
\begin{equation}\label{eqDsymmap}
D((\sigma_1 x)_{j_1}, \ldots, (\sigma_\ell x)_{j_\ell}) := \sum_{k \neq j_1} ((\sigma_1 x)_{k}, \ldots, (\sigma_\ell x)_{j_\ell}) + \ldots + \sum_{k \neq j_\ell} ((\sigma_1 x)_{j_1}, \ldots, (\sigma_\ell x)_{k}).
\end{equation}

\begin{lem} \label{lemcor}
$D$ is a fixed-point free effective symmetric correspondence on $C_{\sigma_1, \ldots , \sigma_\ell}$ of degree $\ell (n-1)$.
\end{lem}

\begin{proof}
As a curve in $C\times C$, we have
\[
\begin{split}
D:= \{ (((\sigma_1 x)_{j_1}, \ldots, (\sigma_\ell x)_{j_\ell}), ((\sigma_1 x)_{k_1}, \ldots, (\sigma_\ell x)_{k_\ell})) \in C_{\sigma_1, \ldots , \sigma_\ell} \times C_{\sigma_1, \ldots , \sigma_\ell} \mid \\
k_i = j_i \mbox{ except for exactly one index } i \}.
\end{split}
\]
This description shows that $D$ is an 
effective symmetric correspondence of degree $\ell (n-1)$ of $C_{\sigma_1, \ldots , \sigma_\ell}$. Since $f: Y \ra X$ is \'etale, the correspondence $D$ is fixed-point free.
\end{proof}

\subsection{} From now on, to alleviate the notation, we write
\[
C := C_{\sigma_1, \ldots , \sigma_\ell}.
\]
Also, for each subset $J = \{ i_1, \ldots, i_r \} \subset \{ 1, \ldots, \ell \}$, we put $C_J := C_{\sigma_{i_1}, \ldots , \sigma_{i_r}}$ and, for any two subsets $I,J$ such that $I\subset J$, we let
\[
\pi_J : C \twoheadrightarrow C_J, \quad \pi_{I,J} : C_J \twoheadrightarrow C_I
\]
be the natural projections. For $J = \varnothing$, we have $C_J = C_\varnothing = X$ and $\pi_\varnothing : C \ra X$ is the unique map to $X$. For $J = \{ i \}$, we have $C_J = Y_{\sigma_i}$ and $\pi_J = \pi_{\{i\}} : C \ra Y_{\sigma_i}$ is the i-th projection of the fiber product. Note that the correspondence $D$ induces endomorphisms of all the subquotients $\pi_J^* JC_J / \sum_{I \subset J} \pi_I^* JC_I$ of $JC$. We have

\begin{thm}\label{thmDsubq}

For each $r \in \{0, 1, \ldots , \ell \}$ and each subset $J = \{ i_1, \ldots, i_r \} \subset \{ 1, \ldots, \ell \}$ of cardinality $r$, the endomorphism induced by the correspondence $D$ on the subquotient $\pi_J^* JC_J / \sum_{I \subset J} \pi_I^* JC_I$ is equal to multiplication by $n(\ell -r) -\ell$ .

\end{thm}

\begin{proof}
Without loss of generality, we may assume that $J = \{1, \ldots, r\}$. Consider a general point of $C_J$. We can choose the indices in such a way that this is $((\sigma_1 x)_1, \ldots , (\sigma_r x)_1)$. We have
$$
\pi_J^* \left((\sigma_1 x)_1, \ldots , (\sigma_r x)_1\right) = \sum_{1 \leq i_{r+1}, \ldots, i_\ell \leq n}\left((\sigma_1 x)_1, \ldots , (\sigma_r x)_1, (\sigma_{r+1} x)_{i_{r+1}}, \ldots, (\sigma_\ell x)_{i_\ell}\right)
$$
and
\[
\begin{split}
D\pi_J^* \left( (\sigma_1 x)_1, \ldots , (\sigma_r x)_1 \right) =
\sum_{1 \leq i_{r+1}, \ldots, i_\ell \leq n}
\left(
\sum_{2 \leq k \leq n} \left((\sigma_1 x)_k, \ldots , (\sigma_r x)_1, (\sigma_{r+1} x)_{i_{r+1}} , \ldots , (\sigma_{\ell} x)_{i_\ell}\right) + 
\ldots \right. \\
\ldots + \sum_{2 \leq k \leq n} \left((\sigma_1 x)_1, \ldots , (\sigma_r x)_k,
(\sigma_{r+1} x)_{i_{r+1}}, \ldots, (\sigma_{\ell} x)_{i_{\ell}} \right) \\
+ \sum_{1 \leq k \leq n, k \neq i_{r+1}} \left( (\sigma_1 x)_1, \ldots, (\sigma_r x)_1,
(\sigma_{r+1} x)_k, \ldots, (\sigma_\ell x)_{i_{\ell}} \right) + \ldots \\
\ldots + \left. \sum_{1 \leq k \leq n, k \neq i_{\ell}}
\left( (\sigma_1 x)_1, \ldots , (\sigma_r x)_1,
(\sigma_{r+1} x)_{i_{r+1}}, \ldots, (\sigma_{\ell} x)_k \right) \right),
\end{split}
\]
hence
\[
\begin{split}
(D + r )\pi_J^* \left( (\sigma_1 x)_1, \ldots , (\sigma_r x)_1 \right) =
\sum_{1 \leq k, i_{r+1}, \ldots, i_\ell \leq n}
\left(
\left((\sigma_1 x)_k, \ldots, (\sigma_r x)_1, (\sigma_{r+1} x)_{i_{r+1}}, \ldots,(\sigma_{\ell} x)_{i_\ell}\right) + 
\ldots \right. \\
\ldots + \left. \left((\sigma_1 x)_1, \ldots , (\sigma_r x)_k,
(\sigma_{r+1} x)_{i_{r+1}}, \ldots, (\sigma_{\ell} x)_{i_{\ell}} \right) \right) \\
+ \sum_{1 \leq i_{r+1}, \ldots, i_\ell \leq n}
\left(
\sum_{1 \leq k \leq n, k \neq i_{r+1}} \left( (\sigma_1 x)_1, \ldots, (\sigma_r x)_1,
(\sigma_{r+1} x)_k, \ldots, (\sigma_\ell x)_{i_{\ell}} \right) + \ldots \right. \\
\ldots + \left. \sum_{1 \leq k \leq n, k \neq i_{\ell}}
\left( (\sigma_1 x)_1, \ldots , (\sigma_r x)_1,
(\sigma_{r+1} x)_{i_{r+1}}, \ldots, (\sigma_{\ell} x)_k \right) \right).
\end{split}
\]
Now note that
\[
\begin{split}
\sum_{1 \leq k, i_{r+1}, \ldots, i_\ell \leq n}
\left(
\left((\sigma_1 x)_k, \ldots, (\sigma_r x)_1, (\sigma_{r+1} x)_{i_{r+1}}, \ldots,(\sigma_{\ell} x)_{i_\ell}\right) + 
\ldots \right. \\
\ldots + \left. \left((\sigma_1 x)_1, \ldots , (\sigma_r x)_k,
(\sigma_{r+1} x)_{i_{r+1}}, \ldots, (\sigma_{\ell} x)_{i_{\ell}} \right) \right)
\end{split}
\]
belongs to $\sum_{I \subset J} \pi_I^* \Pic C_I$, while
\[
\begin{split}
\sum_{1 \leq i_{r+1}, \ldots, i_\ell \leq n}
\left(
\sum_{1 \leq k \leq n, k \neq i_{r+1}} \left( (\sigma_1 x)_1, \ldots, (\sigma_r x)_1,
(\sigma_{r+1} x)_k, \ldots, (\sigma_\ell x)_{i_{\ell}} \right) + \ldots \right. \\
\ldots + \left. \sum_{1 \leq k \leq n, k \neq i_{\ell}}
\left( (\sigma_1 x)_1, \ldots , (\sigma_r x)_1,
(\sigma_{r+1} x)_{i_{r+1}}, \ldots, (\sigma_{\ell} x)_k \right) \right) \\
= (n-1)(l-r) \pi_J^* \left( (\sigma_1 x)_1, \ldots , (\sigma_r x)_1 \right).
\end{split}
\]
\end{proof}

\begin{cor}\label{corPol}
$D$ satisfies the following completely decomposable equation
\begin{equation} \label{eq3.1}
\prod_{r=0}^\ell (D-nr +\ell) = 0
\end{equation}
where we identify any integer $r$ with $r$ times the identity correspondence (i.e., $r$ times the diagonal).
Furthermore, for each $r$, the eigen-abelian subvariety $A_{n(\ell-r) -\ell}$ of $JC$ for the eigenvalue $n(\ell-r) -\ell$ is contained in $\sum_{\#J =r} \pi_J^* JC_J$ and its projection to the quotient $\sum_{\#J =r} \pi_J^* JC_J / \sum_{\#I \leq r-1} \pi_I^* JC_I $ is an isogeny.

In particular, 
\begin{enumerate}
\item for $r=0$, we have $A_{n\ell -\ell} = \pi_\varnothing^* JX$,
\item $A_{n(\ell -1) -\ell}$
is isogenous to
\[
\left( \sum_{i=1}^\ell \pi_{\{i\}}^* JY \right) / \pi_\varnothing^* JX,
\]
\item and
\[
\sum_{r=1}^\ell A_{n(\ell -r) -\ell}
\]
is the Prym variety of the cover $\pi_\varnothing : C \ra X$.
\end{enumerate}
\end{cor}

\subsection{} It follows in particular that $A_{-\ell}$ is the ``new'' part of $JC$, meaning the part that is complementary to the images of Jacobians of curves $C_{\sigma_{i_1}, \ldots, \sigma_{i_r}}$ for $r < \ell$ which, by the above, is equal to $\sum_{r=0}^{\ell -1} A_{n(\ell -r) -\ell}$. Let $A_{-\ell}'$ be the quotient of $JC$ by the abelian subvariety $\sum_{r=0}^{\ell -1} A_{n(\ell -r) -\ell}$. Let $\mu : A_{-\ell} \inj JC \surj A_{-\ell}'$ be the restriction of the polarization of $JC$ to $A_{-\ell}$ and let $m$ be a positive integer such that the kernel of $\mu$ is contained in the $m$-torsion subgroup $A_{-\ell}[m]$ of $A_{-\ell}$. Choose a subgroup $K$ of $A_{-\ell}'[m]$, maximal isotropic with the respect to the Riemann form of $\mu$. Then the polarization $\mu$ induces a principal polarization $\Xi$ on the quotient $B := A_{-\ell}' / K$. By \cite[Proposition 1.7]{Welters1987-1}, the image of an Abel embedding of $C$ in $B$ has cohomology class $m \frac{[\Xi]^{d-1}}{(d-1)!}$ where $d$ is the dimension of $B$. Note that the intersection $A_{-\ell} \cap \sum_{r=1}^\ell A_{nr - \ell}$ is contained in $A_{-\ell}[m]$ (see, e.g. \cite[p. 88]{Welters1987-1}). We have

\begin{lem}\label{lemclass}
We have
\[
A_{-\ell} \cap \sum_{r=1}^\ell A_{nr - \ell} \subset A_{-\ell}[\ell ! n^\ell].
\]
In particular, we can choose $B$ so that the image of an Abel embedding of $C$ in $B$ has cohomology class $\ell ! n^\ell \frac{[\Xi]^{d-1}}{(d -1)!}$
\end{lem}
\begin{proof}
Suppose that $x \in A_{-\ell} \cap \sum_{r=1}^\ell A_{nr - \ell}$. So we can write
\[
x = \sum_{r=1}^\ell x_r
\]
where $x_r \in A_{nr -\ell}$. Successively applying the endomorphisms $D +\ell - rn$ for $r= 0, \ldots, \ell$ gives the series of equations
\[
\begin{array}{lllll}
0 = n x_1 & + 2n x_2 & + 3n x_3 + \ldots & + (\ell -1) n x_{\ell -1} & + \ell n x_\ell \\
0 = & 1 \cdot 2 n^2 x_2 & + 2 \cdot 3 n^2 x_3 + \ldots & + (\ell -2)(\ell -1) n^2 x_{\ell -1} & + (\ell -1)\ell n^2 x_\ell \\
\vdots & & & & \\
0 = & & & 1 \cdot \ldots \cdot (\ell -1) n^{\ell -1} x_{\ell -1} & + 2 \cdot \ldots \cdot \ell n^{\ell -1} x_\ell \\
0 =  & & & & 1 \cdot \ldots \cdot \ell n^\ell x_\ell
\end{array}
\]
So we obtain $\ell ! n^\ell x_{\ell} = (\ell -1) ! n^\ell x_{\ell -1} =0$. Going from bottom to top and multiplying the $r$-th equation by $(\ell -r) !$, we obtain
\[
r ! (\ell -r) ! n^\ell x_r =0
\]
for all $r= 1, \ldots, \ell$. In particular,
\[
\ell ! n^\ell x_r = \ell ! n^\ell x =0
\]
for all $r= 1, \ldots, \ell$.
\end{proof}

\begin{hyp}
From now on in this section, we assume that the curve $C_{\sigma_1, \ldots , \sigma_\ell}$ is irreducible.
\end{hyp}

\subsection{} For any integer $r$, put
\[
g_r := n^r (g -1)+1.
\]
Since $C$ is an \'etale cover of degree $n^\ell$ of $X$, it has genus $g_\ell$. For $r=0, \ldots, \ell$, denote
\[
 d_r := \dim A_{rn -\ell}, 
\]
the dimension of the eigen-abelian subvariety of $JC$ for the eigenvalue $rn -\ell$.
In particular,
\[
d= d_0 = \dim A_{-\ell}, \quad d_\ell = g, \hbox{ and } \sum_{r=0}^\ell d_r = g_\ell.
\]
By \cite[Prop. 11.5.2 p. 334]{cav} and Lemma \ref{lemcor}, we have the trace formula
\[
\sum_{r=0}^\ell (nr -\ell) d_r = l(n-1).
\]
Simple manipulations give the following relations
\begin{eqnarray}
\label{eqn1} d_0 + \ldots + d_{\ell -1} & = & g_\ell - g = (n^\ell -1)(g-1), \\
\label{eqn2} d_1 + 2 d_2 + \ldots + (\ell -1) d_{\ell -1} & = & \ell (g_{\ell -1} - g) = \ell (n^{\ell -1} -1)(g-1).
\end{eqnarray}
When $\ell =2$, we immediately compute
\begin{equation}\label{eqnell=2}
d_0 = (n-1)^2 (g-1), \quad d_1 = 2(n-1) (g-1).
\end{equation}
Theorem \ref{thmDsubq} also shows that for each $r \in \{0, \ldots , \ell \}$ and each subset $J = \{ i_1, \ldots, i_{\ell -r} \} \subset \{ 1, \ldots, \ell \}$ of cardinality $\ell - r$, the eigen-abelian subvariety of $JC_J$ for the eigenvalue $-(\ell -r)$ (for the correspondence on $C_J$) is exactly the part of $JC_J$ that maps into the eigen-abelian subvariety of $JC$ for the eigenvalue $nr-\ell$. Therefore, if we denote $d_s(C_J)$ the dimension of the eigen-abelian subvariety of $JC_J$ for the eigenvalue $ns-(\ell -r)$, we have
\[
d_r \leq \sum_{\# J =\ell -r} d_0 (C_J).
\]

\begin{lem}\label{lemdim}
For all $\ell \geq 2$, all $r\in \{0, \ldots, \ell \}$, and for all subsets $J \subset \{ 1, \ldots, \ell \}$ of cardinality $\ell -r$, the dimension $d_s (C_J)$ only depends on $r$. Writing $_{\ell -r}d_s := d_s (C_J)$ for some (or any) $J$, we have, for all $r\in \{ 0, \ldots, \ell -1 \}$,
\[
d_r = {}_\ell{}d_r = {\ell \choose r} (n-1)^{\ell - r}(g-1).
\]
In particular, ${}_\ell{}d_r = {\ell \choose r} {}_{\ell -r}d_0$ and $A_{nr-\ell}$ is isogenous to the product of the eigen-abelian subvarieties of $JC_J$ for the eigenvalue $-(\ell -r)$, where $J$ runs over all subsets of cardinality $\ell - r$ of $\{ 1, \ldots, \ell \}$.
\end{lem}

\begin{proof}
We proceed by induction. We first note that the first step of the induction, for $\ell =2$, was done in \eqref{eqnell=2}. Assume now $\ell \geq 3$ and that the statement holds for all $k \leq \ell -1$. As we saw above, for $r \in \{ 0, \ldots, \ell \}$,
\[
d_r = {}_\ell{}d_r \leq \sum_{\# J =\ell -r} d_0 (C_J).
\]
Hence, for $r \in \{ 1, \ldots, \ell-1 \}$, applying the induction hypothesis to the Jacobians of the curves $C_J$, for all $J$ of cardinality $r$, we have $d_0 (C_J) = {}_{\ell -r}d_0 = (n-1)^{\ell - r}(g-1)$, and
\[
_\ell{}d_r \leq {\ell \choose r} (n-1)^{\ell - r}(g-1).
\]
Now we compute
\[
\sum_{r=1}^{\ell -1} r{\ell \choose r} (n-1)^{\ell - r}(g-1) = \ell (g-1) \sum_{r=1}^{\ell -1} {\ell -1 \choose r-1} (n-1)^{\ell -1} = \ell (n^{\ell -1} -1)(g-1),
\]
which is equal to the sum $_\ell{}d_1 + 2 {}_\ell{}d_2 + \ldots + (\ell -1) {}_\ell{}d_{\ell -1}$ by \eqref{eqn2}. Since, in addition, all the ${}_\ell{}d_r$ and ${\ell \choose r} (n-1)^{\ell - r}(g-1)$ are positive integers and ${}_\ell{}d_r \leq {\ell \choose r} (n-1)^{\ell - r}(g-1)$, we conclude
\[
{}_\ell{}d_r = {\ell \choose r} (n-1)^{\ell - r}(g-1)
\]
for $r\in \{ 1, \ldots, \ell-1 \}$. Now we compute
\[
{}_\ell{}d_0 = (n-1)^\ell (g-1)
\]
using \eqref{eqn1}. The remainder of the lemma now easily follows.
\end{proof}

\subsection{} Note, in particular, that each abelian subvariety $A_r$ of $JC$ is isogenous to the product of at least ${\ell \choose r}$ abelian subvarieties. Hence $JC$ is isogenous to the product of at least $\sum_{r=0}^\ell {\ell \choose r} = 2^\ell$ abelian subvarieties.




\section{Special cases}\label{secspec}

\subsection{} For $\ell =1$, the curve $C = C_{\sigma_1}$ coincides with the \'etale cover $Y_{\sigma_1} = Y \stackrel{\sigma_1 f}{\lra} X$ and the correspondence $D$ is the usual correspondence
\[
D = \{ (p,q) \mid p \neq q, f(p) = f(q) \} \subset Y^2,
\]
residual to the diagonal of $Y^2$ in the inverse image of the diagonal of $X^2$. It satisfies the equation
\[
(D+1)(D-n+1) =0
\]
and the eigen-abelian subvariety $A_{-1}$ is the usual Prym variety of the cover $\sigma_1 f : Y \ra X$. As we saw in Corollary \ref{corPol}, $A_{n-1} = \pi_\varnothing^* JX$.

\subsection{} For $\ell =2$, we have
\[
C = \{ ((\sigma_1 x)_i ,(\sigma_2 x)_j) \mid i,j = 1, \dots,n \} \subset Y^{2}
\]
and the equation
\[
(D+2)(D-n+2)(D-2n+2) =0.
\]
The abelian subvariety $A_{-2} + A_{n-2} \subset J  C$ is the Prym variety of $\pi_\varnothing : C \ra X$, and $A_{2n-2} = \pi_\varnothing^* JX$.

\subsection{} If $X$ is general, it has no non-trivial automorphisms, but we can still construct the curve $C$ for any \'etale cover $Y$ and $\sigma_1 = \ldots = \sigma_\ell = Id$, the identity. In this case each eigen-abelian subvariety $A_r$ is isogenous to the ${\ell \choose r}$-th power of an abelian subvariety and $JC$ has many repeated factors.

\section{Irreducibility of the curves $C_{\sigma_1, \ldots, \sigma_\ell}$}\label{secirred}

We investigate the irreducibility of the curve $C = C_{\sigma_1, \ldots, \sigma_\ell}$. We start with

\begin{lem}\label{lemirred}
The curve $C$ is irreducible if and only if the monodromy group of the cover $f : Y \ra X$ is $\ell$-transitive.
\end{lem}
\begin{proof}
The datum of a path $\gamma$ in $C$ starting at a point $((\sigma_1 x)_{i_1}, \ldots, (\sigma_\ell x)_{i_\ell})$ is equivalent to the data of paths $\beta_j$ in $Y$ starting at $(\sigma_j x)_{i_j}$ such that $\sigma_1 f(\beta_1) = \ldots = \sigma_\ell f(\beta_\ell)$. Therefore, given a loop $\alpha$ in $X$ based at a point $x$, lifting
it to a path in $C$ means lifting it to paths $\beta_j$ starting at lifts $(\sigma_j x)_{i_j}$ of
$\sigma_j x$ such that $\alpha = \sigma_1 f(\beta_1) = \ldots = \sigma_\ell f(\beta_\ell)$. Therefore, denoting the permutation associated to a loop by the same symbol, the action of $\alpha$ on the fiber of $\pi_\varnothing : C \ra X$ at $x$ is
\[
\alpha \left(\left( (\sigma_1 x)_{i_1}, \ldots, (\sigma_\ell x)_{i_l}\right)\right) = \left(\sigma_1 \circ \alpha ((\sigma_1 x)_{i_1}), \ldots, \sigma_\ell \circ \alpha  ((\sigma_\ell x)_{i_\ell})\right).
\]
Hence, identifying the fundamental groups $\pi_1 (X, \sigma_1 x), \ldots \pi_1 (X, \sigma_\ell x)$ via the maps $\alpha \mapsto \sigma_j \sigma_1^{-1} \circ \alpha$, we conclude that the permutation action of the fundamental group of $X$ on the fiber of $\pi_\varnothing$
at $x$ is transitive if and only if the action of the fundamental
group of $X$ on the fiber of $f$ at $x$ is $\ell$-transitive.\end{proof}

Note that the above result shows that the irreducibility of $C$ only depends on the cover $f : Y \ra X$.

Clearly, $\mathfrak{S}_n$ itself is $n$-transitive.

Multiply transitive groups fall into six infinite families and four classes of sporadic groups. They are all primitive, i.e., they do not preserve any partitions of $\{ 1, \ldots, n\}$.

Below $q$ is a power of a prime number. We refer to \cite{DixonMortimer1996} for the following.

\begin{itemize}

\item For $\ell\geq 6$, the only $\ell$-transitive groups are the symmetric and alternating groups on $\ell$ and $\ell +2$ letters respectively.

\item The solvable $2$-transitive groups are subgroups of the group of affine transformations $AGL_r(p)$ for some integer $r$ and prime number $p$ which contain all translations. The list of these groups can be found in the Appendix to \cite{Liebeck1987} and is attributed to Hering \cite{Hering1985}.

\item The insoluble $2$-transitive groups are almost simple, i.e., they contain a non-abelian simple group and are contained in the automorphism group of that simple group. The insoluble $2$-transitive groups can be found in \cite{CurtisKantorSeitz1976} and \cite{PraegerSoicher1997}.

\item The projective special linear groups $PSL(d,q)$ are $2$-transitive except for the special cases $PSL(2,q)$ with $q$ even, which are $3$-transitive.

\item The symplectic groups defined over the field of two elements have two distinct actions which are $2$-transitive.

\item The field $K$ of $q^2$ elements has an involution $\sigma(a)=a^q$, which allows a Hermitian form to be defined on a vector space on $K$. The unitary group on $V:= K^{\oplus 3}$, denoted $U_3(q)$, preserves the isotropic vectors in $V$. The action of the projective special unitary group $PSU(q)$ is $2$-transitive on the isotropic vectors.

\item The Suzuki group of Lie type $Sz(q)$ is the automorphism group of an $S(3,q+1,q^2+1)$ Steiner system, an inversive plane of order $q$, and its action is $2$-transitive.

\item The Ree group of Lie type $R(q)$ is the automorphism group of an $S(2,q+1,q^3+1)$ Steiner system, a unital of order $q$, and its action is $2$-transitive.

\item The Mathieu groups $M_{12}$ and $M_{24}$ are the only $5$-transitive groups besides $S_5$ and $A_7$. The groups $M_{11}$ and $M_{23}$ are $4$-transitive, and $M_{22}$ is $3$-transitive.

\item The projective special linear group $PSL(2,11)$ has another $2$-transitive action related to the Witt geometry $W_{11}$.

\item The Higman-Sims group $HS$ is $2$-transitive.

\item The Conway group $Co_3$ is $2$-transitive.

\item Other $3$-transitive groups include $PSL(2,7):2$ acting on $8$ items, as generated by the permutations $(a,b,c,d)(e,f,g,h), (a,f,c)(d,e,g)$, and $(e,f)(d,h)(b,c)$; and $PSL(2,11):2$ acting on $12$ items, as generated by the permutations $(g,b,c,i,d)(j,e,h,f,l)$, $(a,b,c)(d,e,f)(g,h,i)(j,k,l)$, and $(a,i)(d,g)(e,j)(h,k)(c,f)$.

\end{itemize}

Choosing covers with the above monodromy groups will produce many examples of curves $C$ that are irreducible.


\begin{thebibliography}{999999}

   \bibitem[BL04]{cav}
C.~Birkenhake and H.~Lange, \emph{Complex abelian varieties}, Grundlehren der
  Mathematischen Wissenschaften [Fundamental Principles of Mathematical
  Sciences], vol. 302, Springer, Berlin, 2004.
\bibitem[CLRR09]{CLRR2009}
Angel Carocca, Herbert Lange, Rub\'{\i}~E. Rodr\'{\i}guez, and Anita~M. Rojas,
  \emph{Prym-{T}yurin varieties using self-products of groups}, J. Algebra
  \textbf{322} (2009), no.~4, 1251--1272.

\bibitem[CKS76]{CurtisKantorSeitz1976}
Charles~W. Curtis, William~M. Kantor, and Gary~M. Seitz, \emph{The 2-transitive
  permutation representations of the finite {C}hevalley groups}, Trans. Amer.
  Math. Soc. \textbf{218} (1976), 1--59.

\bibitem[DM96]{DixonMortimer1996}
John~D. Dixon and Brian Mortimer, \emph{Permutation groups}, Graduate texts in
  Math., vol. 163, Springer, Berlin, 1996.

\bibitem[Her85]{Hering1985}
Christoph Hering, \emph{Transitive linear groups and linear groups which
  contain irreducible subgroups of prime order, {II}}, J. of Algebra
  \textbf{93} (1985), 151--164.


\bibitem[Lie87]{Liebeck1987}
Martin~W. Liebeck, \emph{The affine permutation groups of rank 3}, Proc. London
  Math. \textbf{54} (1987), no.~3, 477--516.

\bibitem[PS97]{PraegerSoicher1997}
Cheryl~E. Praeger and Leonard~H. Soicher, \emph{Low rank representations and
  graphs for sporadic groups}, Cambridge University Press, Cambridge, UK, 1997.

\bibitem[Wel87]{Welters1987-1}
G.~E. Welters, \emph{Curves of twice the minimal class on principally polarized
  abelian varieties}, Nederl. Akad. Wetensch. Indag. Math. \textbf{49} (1987),
  no.~1, 87--109.

\end{thebibliography}

\end{document}